\documentclass[journal]{IEEEtran}
\usepackage{enumerate}
\usepackage{amsthm}
\usepackage{amsmath}
\usepackage{amscd}
\usepackage{amssymb}
\usepackage{latexsym}
\usepackage{array}
\usepackage{tabularx}
\usepackage{setspace}
\usepackage[hyphens,spaces]{url}
\usepackage{cite}
\usepackage{multirow}
\usepackage[table]{xcolor}

\usepackage{graphicx,subfigure}
\usepackage{bbm}
\usepackage[numbers, square, comma, compress]{natbib}
\newtheorem{lem}{Lemma}
\newtheorem{corol}{Corollary}
\newtheorem{theorem}{Theorem}

\theoremstyle{definition}
\newtheorem{rem}{Remark}

\begin{document}

\title{Wang Algebra: From Theory to Practice}

\author{Bob Ross, \textit{Life Member, IEEE}, and Cong Ling, \textit{Member, IEEE}
\thanks{
{This work was presented in part in Asian IBIS Summit 2007 and Virtual European IBIS Summit 2022.}

B. Ross is with Teraspeed Labs, 10238 SW Lancaster Road, 97219, USA, (e-mail: bob@teraspeedlabs.com).

C.  Ling  is  with  the  Department  of  Electrical  and  Electronic  Engineering, Imperial College London, London SW7 2AZ, UK (e-mail: cling@ieee.org).
	}
}	

\maketitle

%

\begin{abstract}
Wang algebra was initiated by Ki-Tung Wang as a short-cut method for the analysis of electrical networks. It was later popularized by Duffin and has since found numerous applications in electrical engineering and graph theory. This is a semi-tutorial paper on Wang algebra, its history, and modern applications. We expand Duffin's historic notes on Wang algebra to give a full account of Ki-Tung Wang's life. A short proof of Wang algebra using group theory is presented.  We exemplify the usefulness of Wang algebra in the design of T-coils. Bridged T-coils give a significant advantage in bandwidth, and were widely adopted in Tektronix oscilloscopes, but design details were guarded as a trade secret. The derivation presented in this paper, based on Wang algebra, is more general and simpler than those reported in literature. This novel derivation has not been shared with the public before.

\end{abstract}


\section{Introduction}

Wang algebra is a commutative algebra $\mathbb{W}$ with the properties
\begin{equation}\label{eq:Wang-rule}
    x+x = 0, \ \ x^2 = 0
\end{equation}
for all $x \in \mathbb{W}$. It was proposed by Ki-Tung Wang in 1934 as a convenient rule to simplify the analysis of electrical networks \cite{Wang}. In essence, it is a clever method to compute the determinant of a symmetric matrix in the context of solving a system of simultaneous linear equations. A usual approach in the pre-computer era would be to use Cramer's rule, but calculating the determinants by hand is not only tedious, but also highly likely to run into mistakes, due to the many terms involved. In Wang algebra, a large number of terms vanish due to \eqref{eq:Wang-rule}, thus greatly reducing the calculation complexity \cite{Ting,Tsai,Chow,KU,BELLERT}. Wang algebra has been a useful method to design electronic circuits and modern integrated circuits (ICs), in particular T-coils and interconnects \cite{Ross94,Ross07}. Yet, the impact of Wang algebra goes much beyond electrical engineering. It has also found applications in networking and graph theory \cite{Chen-Graph}. Historically, Wang algebra was studied in great detail by famous mathematicians Wei-Liang Chow \cite{Chow} and Richard Duffin \cite{Duffin,Duffin-Morley}\footnote{The works by Chow \cite{Chow} and Duffin \cite{Duffin} themselves deserve special remarks.

Wei-Liang Chow, who had studied algebraic geometry under the great algebraist van der Waerden, was running a business in Shanghai, China, to support his family. Shanghai was occupied by Japanese troops during World War II, and as a result he had lost his academic job in the National Central University. Chow could hardly do any research during this war period (1937-1945), and the study of Wang algebra \cite{Chow} was an exception (and in fact his only work on an engineering problem). Amazingly, after almost a decade in business, he made a strong return to research and became a prominent algebraic geometer in the US. Chow ring, a fundamental concept in algebraic geometry, is named after him.

Richard Duffin is well known for his contributions to electrical networks and geometric programming. Together with Albert Schaeffer, he introduced the concept of frames, which play an important role in signal processing. The Duffin-Schaeffer conjecture, concerning rational approximation of real numbers, is a famous conjecture in number theory that had stood open for nearly 80 years. It was finally proved by Dimitris Koukoulopoulos and James Maynard in 2019. Maynard won a Fields Medal partly because of this proof.}.

Somewhat mysteriously, Duffin gave the following comment in \cite{Duffin74}:

``K. T. Wang managed an electrical power plant in China, and in his spare time sought simple rules for solving the network equations... Wang could not write in English so his paper was actually written by his son, then a college student."

However, the comment by Duffin seemed incomplete and did not convey Ki-Tung Wang's full story. The first aim of this article is to provide an expanded biography of Ki-Tung Wang.

\subsection{Biography of Ki-Tung Wang}

\begin{figure}
  \centering
  \includegraphics[width=8cm]{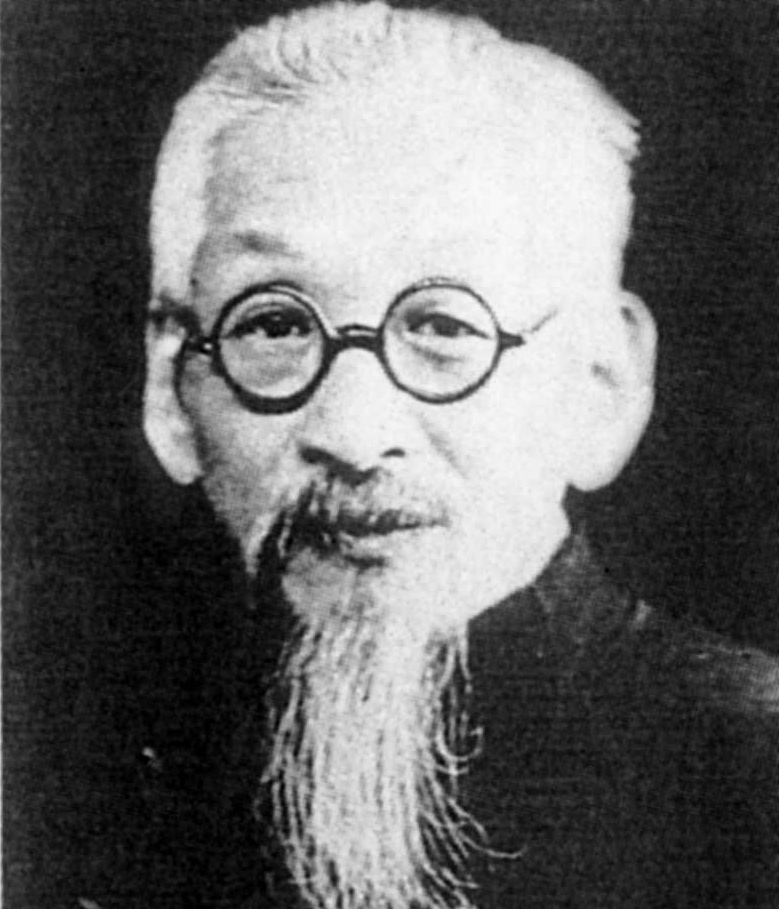}\\
  \caption{Ki-Tung Wang (1875-1948), Chinese mathematician, electrical engineer and philosopher. Photo courtesy of Dr Jin-Hai Guo, Institute for the History of Natural Sciences, Chinese Academy of Sciences.}\label{fig:KTWang}
\end{figure}

Ki-Tung Wang (1875-1948),
shown in Fig.~\ref{fig:KTWang}, was a Chinese mathematician, electrical engineer and philosopher. Believed to be the first Chinese scholar to publish a mathematical paper in an international journal, he is well known for his work on Wang algebra, as well as investigation on the relationship between sciences and Buddhism \cite{Guo2015}.

Ki-Tung Wang was born into a prominent family in Suzhou, Jiangsu Province in 1875. His ancestor Ao Wang (1450-1524) was ranked no. 3 in the Imperial Examination and later became a Grand Secretary of the Cabinet (equivalent to Prime Minister) of the Ming Dynasty. His father Song-Wei Wang (1849-1895) was also a Jinshi (Imperial Scholar), the highest degree of Imperial Examination in China.

In 1895, Ki-Tung Wang graduated from the Imperial Tungwen College\footnote{It literally means ``Multilingual College".}
(modern-day Peking University), then was hired as a mathematical lecturer there. He had already published several Chinese articles on mathematics, and investigated the relation between classical Chinese mathematics and modern mathematics \cite{Guo15}.

In 1909, Ki-Tung Wang served as an administrator of Chinese students in Europe, then he did internships at the English Electrical Company and Siemens \cite{Guo2015}. During this period, he published a paper on the differentiation of quaternionic functions\footnote{Quaternions are extension of complex numbers, with lots of engineering applications such as astronautics, robotics, computer visualization, animation, special effects in movies, navigation, etc.
It is worth mentioning that differentiation of quaternionic functions is highly nontrivial---even more involved than the Cauchy-Riemann condition for complex functions.} in the Proceedings of the Royal Irish Academy \cite{KTWang-Quaternion}, which is believed to be the first paper published by Chinese mathematicians in international journals \cite{Guo02,Guo2015}.

In 1912, Ki-Tung Wang was offered a job at the Ministry of Education, then in as part of the newly founded Republic of China \cite{Guo2015}.

In 1914, Ki-Tung Wang went to industry and became an electrical engineer at the Zhenjiang Power Plant, Jiangsu Province \cite{Guo2015}.

In 1928, Ki-Tung Wang was appointed Research Fellow\footnote{Equivalent to Principal Scientist nowadays.} at the National Research Institute of Engineering, Academia Sinica \cite{Guo2015}. He proposed a new method to derive the impedance of electrical networks, which is sometimes advantageous to the traditional method \cite{Wang,Guo03b}. In 1950, Duffin and Bott recognized that his rules form an algebra and presented this method to the American Mathematical Society, under the title ``The Wang algebra of networks" \cite{Duffin}\footnote{Duffin himself cited in \cite{Duffin} H. W. Becker's 1948 unpublished notes on Wang algebra.}. It is unclear how Duffin learned the story that Wang could not write in English\footnote{One possibility is that Duffin might have met Ki-Tung Wang's daughter or son. Several children of Ki-Tung Wang earned Ph.D. degrees in the US; in particular, his son Shou-Jin Wang (1904-1984) obtained his Ph.D. in physics from Columbia University, worked at Peking University and retired from MIT Lincoln Lab (https://www.guokr.com/article/441034/). It is likely that Duffin had met Shou-Jin Wang.}. This was unlikely to be true, since Ki-Tung Wang had learned English at the Imperial Tungwen College \cite{Guo03}, worked in Britain, and published an English paper \cite{KTWang-Quaternion} before. Nevertheless, it is possible that Ki-Tung Wang needed some help on English writing, since it had been more than 20 years since the publication of \cite{KTWang-Quaternion}.

Ki-Tung Wang was very interested in philosophy beyond the limits of modern sciences. He had several publications on the relationship between sciences and Buddhism, including a book \textit{Comparative Study of Buddhism and Sciences} \cite{WangBook}.

\subsection{From Old Theory to Modern Applications}

Nearly 90 years have passed since Wang algebra was proposed. It has found numerous applications in areas ranging from graph theory to electrical and electronic engineering. This article serves as a survey of the history, theory and modern applications of Wang algebra.

As an example of applications, we will focus on T-coils, which have been used for many years for wideband amplifier designs with distributed amplification, peaking\footnote{In this paper, peaking refers to the technique to improve bandwidth by using passive elements (such as a small inductor or a T-coil).  See \cite[Section 12.3]{Lee} for more details.} networks, termination networks, and transmission line equalization \cite{Ginzton,HP,Lee,Staric,Paramesh,Roy,Feucht,Addis,Battjes,Hollister,Selmi,True,Ross94,Galal,Pillai,Pillai-patent}. In recent years, high-speed IC designs have added T-coils because of three useful properties: 1) constant input resistance (constant-R), 2) usually over double the bandwidth extension, and 3) a second order transfer function.

Tektronix pioneered the use of T-coils for bandwidth enhancement in 1960s. Bridged T-coils give a significant advantage in bandwidth over conventional RC design (more precisely $2\sqrt{2}$ times improvement). They were widely adopted in Tektronix oscilloscopes, but design details were guarded as a trade secret \cite{Razavi,Lee,RAKO}. It was not until 1990 that former Tektronix engineer Feucht provided the T-coil design equations \cite{Feucht}, but no derivation was given. Lee published a tedious derivation for the standard symmetric T-coil \cite[Chap. 12]{Lee}.

In this paper, a general method is presented to design constant impedance bridged-T networks in terms of symbolic elements.  General formulas for the element relationships are derived using Wang Algebra. The constant impedance constraint is used by adding an element to balance a bridge circuit.  This simplifies calculating the transfer function, and it is presented in a reduced form in terms of a Th\'evenin equivalent circuit. When applied to a constant input resistance (constant-R) network, such as a bridged T-coil that connects to a capacitor load, extended designs with optional resistors in series with inductors and resistors in parallel with capacitors exist for both symmetrical and asymmetrical configurations. The process is to substitute known element values in the symbolic formulas and to solve for the remaining elements. All transfer functions are second order, and it is convenient to parameterize the formula in terms of a bridging capacitor.
Compared with an RC circuit, the bandwidth is usually more than doubled with selected peaking. Previously published and new configurations can be derived from these formulas.

\subsection{Organization}

The rest of this paper is organized as follows. In Section~\ref{sect:2}, we give an introduction to Wang's algebra, as well as a one-page proof based on group theory. Section~\ref{sect:3} is a survey of the applications of Wang's algebra to electrical engineering and graph theory. Section~\ref{sect:4} is devoted to the design of T-coils using Wang algebra. Conclusions are given in Section~\ref{sect:conclusion}.

\section{Wang Algebra}\label{sect:2}

In essence, Wang algebra gives a method to calculate the determinant of a symmetric matrix, which can be more convenient sometimes. A proof was already outlined in the original article \cite{Wang}. Duffin's proof \cite{Duffin} was based on Grassmann algebra, which is 10-pages long. Chow's proof \cite{Chow} based on matrix theory is also quite tricky. In the following, we will present a short proof using group theory, for completeness. Familiarity with group theory is assumed, in particular symmetric group $S_n$ of order $n$ \cite{HuGZ}. Readers uninterested in the proof may simply skip it.

\begin{theorem}[Wang Algebra]\label{theorem:Wang}
Let $\mathbf{A}=[a_{ij}]_{n\times n}$ be a symmetric matrix, i.e., $a_{ij}=a_{ji}$, where $1\leq i,j \leq n$. Write the diagonal elements of $\mathbf{A}$ as
\[
    a_{ii} = a'_{ii} - \sum_{j\neq i} a_{ij}.
\]
Then the determinant $\det(\mathbf{A})$ can be computed as
\begin{equation}\label{eq:wang-det}
     \det(\mathbf{A}) = \prod_{i=1}^n \left(a'_{ii} - \sum_{j\neq i} a_{ij} \right)
\end{equation}
in Wang algebra $\mathbb{W}$.
\end{theorem}

It should be pointed out that the product \eqref{eq:wang-det} must be computed symbolically,  for it is certainly not true that $\det(\mathbf{A}) = \prod_{i=1}^n a_{ii}$ numerically. In other words, we take $a'_{ii}, a_{ij} \ (j\neq i)$ as symbols $\in \mathbb{W}$ in the computation of \eqref{eq:wang-det}. The key point here is that many terms disappear thanks to the rule of Wang algebra \eqref{eq:Wang-rule}. Once the product \eqref{eq:wang-det} has been computed, numerical values $a_{ij} \in \mathbb{C}$ can be substituted in to find the determinant. Readers are referred to Section~\ref{sect:3} for examples.

\begin{proof}
We will  prove Theorem \ref{theorem:Wang} in two steps. We use a standard formula of the matrix determinant:
\begin{equation}\label{Eq:Leibniz}
    \det(\mathbf{A}) = \sum_{\sigma\in S_n} \mathrm{sgn}(\sigma) \prod_{i=1}^n a_{i\sigma(i)}
\end{equation}
where $\sigma$ is a permutation in the symmetric group $S_n$, and $\mathrm{sgn}(\sigma)=\pm 1$ denotes the signature of $\sigma$.

\begin{lem}\label{Lem:1}
Let $\mathbf{A}=[a_{ij}]_{n\times n}$ be a symmetric matrix. Then, except the term corresponding to the main diagonal, all terms of Eq. (\ref{Eq:Leibniz}) contain either a square or a factor $2$.
\end{lem}

\begin{proof} Firstly, the factor $2$ is due to the symmetry of the matrix: associated with any term $\prod_{i=1}^n a_{i\sigma(i)}$ in Eq. (\ref{Eq:Leibniz}), there is another one
\begin{equation}\label{Eq:term2}
    \prod_{i=1}^n a_{\sigma(i)i} = \prod_{j=1}^n a_{j\sigma^{-1}(j)}.
\end{equation}
Since $a_{i\sigma(i)}=a_{\sigma(i)i}$, the two terms are equal. Moreover, $\mathrm{sgn}(\sigma)=\mathrm{sgn}(\sigma^{-1})$. Thus we obtain the factor $2$ if the two terms are distinct.

The argument above fails if and only if the two terms are in fact the same (\textit{e.g.}, the term $a_{11}a_{22}\cdots a_{nn}$ corresponding to the main diagonal). This happens if and only if
\begin{equation}\label{Eq:perm}
    \sigma = \sigma^{-1},
\end{equation}
\textit{i.e.}, $\sigma^2=1$. This implies $\sigma$ has order $1$ or $2$.

The case of order $1$ corresponds to the main diagonal.

The case of order $2$ consists of one or more disjoint cycles of length $2$, \textit{i.e.}, transpositions, which looks like $\sigma=(i,j)(k,l)$ {etc.} In other words, such a product contains square(s) $a_{ij}a_{ji} = a_{ij}^2$, $a_{kl} a_{lk}=a_{kl}^2$ {etc}.
\end{proof}

\begin{lem}\label{lem:2}
Let the diagonal elements of $\mathbf{A}$ be given by
\begin{equation}\label{eq:structure}
    a_{ii} = a'_{ii} - \sum_{j\neq i} a_{ij}.
\end{equation}
Then any terms of (\ref{Eq:Leibniz}) involving an off-diagonal element are cancelled out.
\end{lem}

\begin{rem}
Lemma~\ref{lem:2} is slightly more general than Wang algebra, since it does not assume that $\mathbf{A}$ is symmetric.
\end{rem}

\begin{proof}
Without loss of generality, consider an off-diagonal element $a_{ij}$ where $i\neq j$.
The sum of all terms involving $a_{ij}$ is given by
\begin{equation}\label{Eq:Sum-aij}
    \sum_{\sigma: \sigma(i)=j} \mathrm{sgn}(\sigma) a_{ij} \prod_{k=1,k\neq i}^n a_{k\sigma(k)}.
\end{equation}
We claim that all these terms will be cancelled out by those in the following sum
\begin{equation}\label{Eq:Sum-aii}
    \sum_{\sigma': \sigma'(i)=i} \mathrm{sgn}(\sigma') a_{ii} \prod_{k=1,k\neq i}^n a_{k\sigma'(k)}.
\end{equation}

Since $\sigma(i)=j$, indexes $i$, $j$ must belong to a certain cycle $(i_1,i_2,\ldots,i_l)$ of length $l > 1$. Let us extract from (\ref{Eq:Sum-aij}) those containing this cycle:
\begin{align}
   & \sum_{\sigma: (i_1,i_2,\ldots,i_l)\in \sigma} \mathrm{sgn}(\sigma)  \prod_{k \in \{i_1,i_2,\ldots,i_l\}} a_{k\sigma(k)} \prod_{k \notin\{i_1,i_2,\ldots,i_l\}} a_{k\sigma(k)}.   \label{Eq:Cycle-aij}
\end{align}
There exist corresponding terms in (\ref{Eq:Sum-aii}) with $\sigma'$ fixing indexes $i_1,i_2,\ldots,i_l$:
\begin{align}
   & \sum_{\sigma': \sigma'(i_1)=i_1, \cdots, \sigma'(i_l)=i_l} \mathrm{sgn}(\sigma')  \prod_{k \in \{i_1,i_2,\ldots,i_l\}} a_{kk} \prod_{k \notin\{i_1,i_2,\ldots,i_l\}} a_{k\sigma'(k)}. \label{Eq:NoCycle-aij}
\end{align}
Clearly, $\sigma$ is a composition of $\sigma'$ and $(i_1,i_2,\ldots,i_l)$.

Note that a negative counterpart of $a_{k\sigma(k)}$ exists in $a_{kk} = a'_{kk} - \sum_{m\neq k} a_{km}$, so the sign of those terms in (\ref{Eq:NoCycle-aij}) is
\[
\mathrm{sgn}(\sigma')(-1)^l.
\]
But the sign of those terms in is (\ref{Eq:Cycle-aij})
\[
\mathrm{sgn}(\sigma)=\mathrm{sgn}(\sigma')(-1)^{l-1},
\]
they must have different signs, therefore being cancelled out completely.
\end{proof}

The proof of Theorem \ref{theorem:Wang} is completed by combining the two lemmas.
\end{proof}

\begin{rem}
In \cite{Wang}, Ki-Tung Wang treated terms containing a square and terms containing factor 2 separately, which correspond to 2-cycles $(i_1,i_2)$ and longer cycles $(i_1,i_2,\ldots,i_l)$ in the above proof, respectively.
\end{rem}

An astute reader may wonder if there is any gain to use Wang algebra to compute the matrix determinant. After all, applying the formula (\ref{Eq:Leibniz}) crudely requires $n\cdot n!$ multiplications, while  computing (\ref{eq:wang-det}) may require $n^n$ multiplications in the worst case. By Stirling's formula $n!\approx (n/e)^n$, the latter is asymptotically much greater. Nevertheless, there are some advantages of Wang algebra:
\begin{itemize}
    \item The number of terms is greatly reduced by applying Wang's rule (\ref{eq:Wang-rule}), thus it can be more convenient for small dimensions $n$;
    \item Even if there are faster numerical algorithms to compute matrix determinant\footnote{For example, the elimination method (aka Gauss elimination) can be used to compute matrix determinant with $O(n^3)$ complexity. The elimination method was invented by ancient Chinese mathematicians to solve systems of linear equations in \textit{The Nine Chapters on the Mathematical Art} (see \cite{Nine-Chapters} for an English translation), predating the work of Gauss by about 2000 years. \textit{The Nine Chapters on the Mathematical Art} is a textbook taught in the National University during the Han Dynasty of China (202 BCE -- 220 CE). Ki-Tung Wang was aware of the elimination method, since classical Chinese mathematics was taught at the Imperial Tungwen College \cite{Guo03,Guo15}.}, the symbolic formulas given by Wang algebra can lead to considerable insights.
\end{itemize}
We will demonstrate these advantages in the following sections.

%

\section{Applications}\label{sect:3}

\subsection{Applications to Electrical Networks/Circuits}

\begin{figure}
  \centering
  \includegraphics[width=7cm]{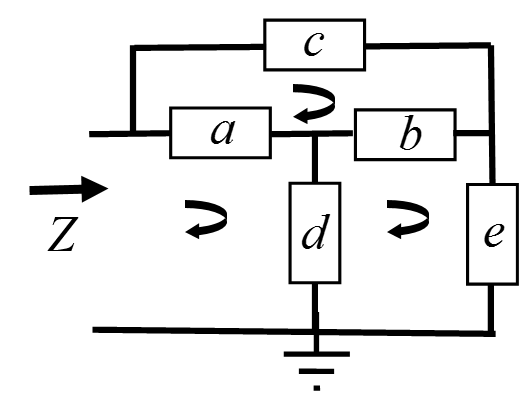}\\
  \caption{Bridged-T network used to derive the constant-Z constraint by loop equations and Wang algebra.}\label{fig:Bridged-T}
\end{figure}

For simplicity, let us assume planar networks, i.e., networks which can be drawn on a plane such that no branches cross each other. An example of planar electrical network is shown in Fig.~\ref{fig:Bridged-T}, with impedances $a, b, c, d, e$. In fact, this is a model of the bridged T-coil which will be analyzed in detail in the following section. The problem is to determine the joint impedance $Z$. Applying Kirchhoff's voltage law to three loops shown, we obtain the following system of linear equations:
\begin{equation}\label{eq:example-3by3}
    \begin{pmatrix}
    a + d & -a & -d \\
    -a  & a + b+c & -b \\
    -d  & -b & b+ d+e
    \end{pmatrix}
    \begin{pmatrix}
    I  \\
    I_2   \\
    I_3
    \end{pmatrix}
    =
    \begin{pmatrix}
    V  \\
    0   \\
    0
    \end{pmatrix}.
\end{equation}
Note that the above matrix is always symmetric and exhibits the structure given in \eqref{eq:structure}, by selecting the orientations of the loops properly. Denote by $\mathbf{M}_1$ the $3\times 3$ matrix on the left-hand side of~(\ref{eq:example-3by3}), and by $\mathbf{M}$ the $2\times 2$ submatrix on its bottom right. Cramer's rule is a standard approach to computing the joint impedance $Z$:
\begin{equation}
    \frac{1}{Z} = \frac{I}{V} = \frac{\det(\mathbf{M})}{\det(\mathbf{M}_1)}.
\end{equation}
To proceed, we firstly compute the numerator:
\begin{align}
    \det{(\mathbf{M})}&= \det{\begin{pmatrix}
     a + b+c & -b \\
     -b & b+ d+e
    \end{pmatrix}} \nonumber \\
    &=ab+ad+ae+b^2+bc+bd+be+cd+ce \nonumber \\
    &=ab+ad+ae+bc+bd+be+cd+ce \label{eq:cotree}
\end{align}
where $b^2=0$ in Wang algebra $\mathbb{W}$. Then we multiply it with $a+d$ to compute the denominator:
\begin{align*}
    \det{(\mathbf{M}_1)} &= (a+d)(ab+ad+ae+bc+bd+be+cd+ce) \\
    &= 2abd+abc+abe+acd+ace+ade+bcd+bde+cde \\
    &=abc+abe+acd+ace+ade+bcd+bde+cde.
\end{align*}
where $2abd=0$ in Wang algebra $\mathbb{W}$. Many product terms are 0 and do not have to be listed. Note the saving it brings in this case: 18 initial terms in the denominator are reduced to 8 in the end.

For the reason to be clear in the next subsection, the determinant $\det{(\mathbf{M})}$ is called the \textit{mesh determinant} of the electrical network.

We can see that another statement of Wang algebra for computing network determinants is:

\begin{corol}[Wang's Rule]
The determinant of a planar network does not contain any terms containing a square or a factor $2$. Moreover, all its terms have coefficient $+1$.
\end{corol}

Ki-Tung Wang initially considered planar networks \cite{Wang}. In follow-up works by Chinese researchers \cite{Ting,Tsai,Chow}, the restriction of planar networks was removed. Wang algebra $\mathbb{W}$ still holds for non-planar networks, but the rule stated above needs to be modified slightly. More precisely, a term coefficient may be an odd number; therefore, the coefficients are calculated modulo $2$ as in Wang algebra $\mathbb{W}$.



Duffin proved that the so-called \textit{just discriminants}, whose coefficients are all equal to $+1$ and which do not necessarily come from Kirchhoffian networks, can also be evaluated in Wang algebra \cite{Duffin}. For example, he showed that Wang algebra holds for a network lying on the surface of a torus. But Grassmann algebra needs to be used for {non-just discriminants}, in a similar way to Wang algebra for just discriminants.

%

\subsection{Applications to Graph Theory}

Of course, electrical networks are examples of graphs. Since matrix theory and graph theory are closely related, it is not surprising that Wang algebra is also useful in graph theory. Specifically, it gives an algebraic method to enumerate the trees and cotrees of a graph. Readers are referred to \cite{Chen-Graph} for an introduction to graph theory and its engineering applications.

A graph consists of a set of nodes together with a set of edges. The electrical network shown in Fig.~\ref{fig:Bridged-T} is an example of graph. It contains 4 nodes, as well as 5 edges labelled by $a$, $b$, $c$, $d$ and $e$.

A (spanning) \textit{tree} of a graph is a set of edges which connect all nodes and which do not contain any loops. For example, edges $\{a,b,e\}$ form a tree in Fig.~\ref{fig:Bridged-T}. The complement of a tree in a graph in called a \textit{cotree}; in other words, it is a set of edges when removed leave no loops. For example, $\{c,d\}$ form a cotree, since it is the complement of $\{a,b,e\}$. Wang algebra gives a handy method to enumerate trees/cotrees.

\subsubsection{Enumerating Trees/Cotrees}

For convenience, let the admittances be $A=1/a$, $B=1/b$, $C=1/c$, $D=1/d$, $E=1/e$ in Fig.~\ref{fig:Bridged-T}. By Kirchhoff's current law, we have another set of equations for Fig.~\ref{fig:Bridged-T}:
\begin{equation}\label{eq:example-KCL}
    \begin{pmatrix}
    A + C & -A & -C \\
    -A  & A + B+D & -B \\
    -C  & -B & B+ C+E
    \end{pmatrix}
    \begin{pmatrix}
    V  \\
    V_2   \\
    V_3
    \end{pmatrix}
    =
    \begin{pmatrix}
    I  \\
    0   \\
    0
    \end{pmatrix}.
\end{equation}

The above matrix also exhibits the structure given in \eqref{eq:structure}. Denote by $\mathbf{S}$ the $3\times 3$ matrix on the left-hand side of~(\ref{eq:example-KCL}), and by $\mathbf{S}_1$ the $2\times 2$ submatrix on its bottom right. The determinant $\det{(\mathbf{S})}$ is called the \textit{node determinant} of the electrical network.

Applying Cramer's rule yields the joint impedance
\begin{equation}
    {Z} = \frac{V}{I} = \frac{\det(\mathbf{S}_1)}{\det(\mathbf{S})}.
\end{equation}
Again, calculation is greatly simplified by using Wang algebra. The numerator can be calculated as
\begin{align*}
    \det{(\mathbf{S}_1)}&= (A + B+D)(B+ C+E)\\
    &=AB+AC+AE+BC+BE+BD+CD+DE.
\end{align*}
Multiplying it with $A+C$ yields the denominator
\begin{align*}
    \det{(\mathbf{S})} &= (A+C)\det{(\mathbf{S}_1)} \\
    &= ABE+ABD+ACD+ADE \\
    & +ACE+BCE+BCD+CDE.
\end{align*}
The saving is the same as before: 18 initial terms in the denominator are reduced to 8 in the end.

A tree, respectively cotree, product is the product of the edges of the tree, respectively cotree. It is known that the terms of node determinant are tree products \cite[Theorem 2.28]{Chen-Graph}, while the terms of mesh determinant are cotree products \cite[Theorem 2.29]{Chen-Graph}.
In fact, the 8 terms of node determinant  $\det{(\mathbf{S})}$  are precisely the trees, while those of the mesh determinant \eqref{eq:cotree} are the cotrees of the graph shown in Fig.~\ref{fig:Bridged-T}.

The following relation holds between the mesh determinant and node determinant:
\begin{equation}
    \det(\mathbf{M}(a,b,c,d,e))= abcde \cdot \det(\mathbf{S}(A,B,C,D,E)).
\end{equation}
This agrees with the one-to-one correspondence between trees and cotrees.

\subsubsection{Counting Trees/Cotrees}

Obviously, setting $a=b=c=d=e=1$, we obtain the number of trees/cotrees of a graph:
\begin{equation}
    \det(\mathbf{M}(1,1,1,1,1))= \det(\mathbf{S}(1,1,1,1,1)).
\end{equation}
For example,
\begin{equation}\label{eq:Lapalacian}
    \mathbf{S}(1,1,1,1,1) = \begin{pmatrix}
    2 & -1 & -1 \\
    -1  & 3 & -1 \\
    -1  & -1 & 3
    \end{pmatrix}.
\end{equation}
It is easy to check that $\det(\mathbf{S}(1,1,1,1,1))=8$, implying there are 8 trees in the graph, as expected. In fact, $\det(\mathbf{S}(1,1,1,1,1))$ is a cofactor of the Laplacian matrix of the graph, and this is a well-known method in graph theory to calculate the number of trees. This fact can be used to check if we have obtained the correct number of terms in network determinants.

\section{Bridged T-coils}

In this section, we demonstrate how Wang algebra can be used to greatly simplify the design of bridged T-coils. T-coils are an old technology, but featured with modern applications. Constant-R T-coils provide ideal load or termination. They offer $2.73$ improvement for acceptable $0.4\%$ overshoot to ideal step input. Now T-coils are used in high-speed buffer design and electrostatic discharge (ESD) compensation with bandwidth improvement. A review of historical applications at Tektronix is given in Section~\ref{sect:4}.

For peaking applications, T-coils are designed to drive capacitive loads that model input transistor structures (or in a reciprocal manner, current source transistor drivers with output capacitance). T-coil extensions for loads with series or parallel resistor elements have been published \cite{Roy,Feucht,Addis,Hollister,Selmi,True,Ross94}. Also, more detailed models of the physical structure reveal resistive losses in the branches \cite{Galal,Pillai,Pillai-patent}. With the formulas in this section, extensions for losses or resistive elements can be included in the design process. Extended T-coil configurations are shown in Section~\ref{sect:II}.

Several methods have been reported to derive the transfer function to a capacitor load \cite{Ginzton,Lee,Staric,Paramesh,Roy}, and this remains an algebraically tedious process, especially for the extended configurations. Most of these methods first derive fourth order equations. These are reduced to second order equations after identifying common factors. The calculations (not shown) are particularly tedious for the general configurations reported later \cite{Ross94}.

For generality, a symbolic method using Wang algebra for simplicity is used in Section \ref{sect:III} to derive the general constant input impedance (constant-Z) constraint relationship. A symbolic approach (based on an augmented Th\'evenin equivalent circuit) is used along with the constant-Z constraint to produce a transfer function that is already of reduced order (or easy to reduce).

Section~\ref{sect:IV} gives two sets of T-coil formulas for the most general symmetrical and asymmetrical cases where the termination is $R$.  Simpler T-coils can be derived by removing some resistor terms. A design process is suggested based on selecting complex-pole angles.

\subsection{T-coil History at Tektronix and Afterwards}\label{sect:4}

The application of T-coils in products is documented in personal narratives by C. R. Battjes \cite{Battjes} and J. L. Addis \cite{Addis}.  The first author adds T-coil derivations done while at Tektronix and after leaving Tektronix.  Around 1948, William Hewlett had lunch in Oregon with Tektronix president Howard Vollum and a key engineer Logan Belleville.  Hewlett penciled out the distributed amplifier circuit \cite{HP} on a paper napkin.  Tektronix adopted distributed amplifiers or distributed deflection circuits in cathode ray tubes in oscilloscopes.  The input for the Tektronix 519 oscilloscope (1 GHz) drove the distributed vertical deflection plates directly.  Other products included the 517 (50 MHz with distributed amplifier), 545 main frame with plugins (30 MHz), 545A (30 MHz), and 585A (100 MHz with distributed deflection plates).  T-coils for inter-stage peaking were used in Type K and L plugins and in the 3A6 plugin that connected directly to a cathode ray tube deflection structure. This is just a short list of some early implementations.  More information is available by searching the product names. Note, some schematics assume, but do not show the bridging capacitor $C_B$ (see Fig.~\ref{fig:Fig1}).


The portable oscilloscopes, Tektronix 454 and 454A (150 MHz) used distributed deflection elements and an output amplifier with T-coil peaking.  It also used T-coils for input delay line differential phase compensation.  T-coils on individual circuit boards were fabricated for transistor interstate peaking.  The transistor input was still approximated as capacitor.  A leading edge (at that time) oscilloscope the Tektronix 7904 (500 MHz) had minimum VSWR T-coil that compensated for some additional parasitic degradations.  The follow-on 7104 (1 GHz) oscilloscope used faster transistors, thin film conductors on substrates, and a transmission line package design.

The so-called ``Ross T-coil" is noted in \cite{Addis,Battjes} and was implemented by thin film deposition to peak and terminate the inter-chip transmission lines. Other T-coils boosted the frequency response and minimized reflections in a $50\Omega$ input for the 11A72 dual channel plugin preamplifier hybrid.

C. R. Battjes \cite{Battjes} (a Stanford University graduate) joined Tektronix to design bipolar transistor amplifiers. He created and conducted internally the amplifier frequency and time response (AFTR) class for new engineers.  The T-coil derivations were those of \cite{HP} or variations thereof. Rather than using a full bipolar transistor hybrid-pi load input model used in Fig. \ref{fig:Fig2}(d) and Figs. \ref{fig:Fig3}(c,d); Battjes favored the simpler $R_S-C$ load model in Fig. \ref{fig:Fig2}(b) and Fig. \ref{fig:Fig3}(a).  It offered better time-constant design intuition and the $R_P$ contribution was considered large and negligible. Besides, the full design equation mathematics shown in this paper were not known at that time.

Sitting next to Battjes, the first author became interested in working out the mathematics for the $R_S-C$ model.  He applied Wang algebra, but never revealed the derivation details outside of Tektronix.  Various forms of the result were later published as the “Ross T-coil” \cite{Addis,Battjes}.

The Wang algebra approach differed from earlier known derivations.  The coupled inductor was split into three inductors, as shown in Fig. \ref{fig:Fig1}, to avoid any pre-defined coupling assumptions.  A constant resistance constraint was applied initially to calculate some fixed element values (shown on and above the Gain line in Table \ref{table}). The equations were helpful in making later simplifications.  The parameter of interest was selected as $C_B$ because it could be varied from 0 to nearly infinity. $C_B$ was easily related to the inductance $M$, the coupling coefficient $k$, and a second order equation damping variable $\delta$.  This derivation did not require taking an absolute value of $k$, as was done in some other derivations.  Nor was it limited to symmetrical T-coils.  (For example, the T. T. True patent \cite{True} for asymmetrical T-coils was known internally, but never used.)

After leaving Tektronix, the first author investigated deriving the equations for the full bipolar transistor hybrid-pi input model in Fig. \ref{fig:Fig2}(d) and Figs. \ref{fig:Fig3}(c,d), as discussed in this paper.  As technology evolved, digital displays replaced distributed deflection structures.  Also, hybrid integrated circuits and integrated circuits were replacing discrete transistor circuits.  But as the references show, T-coils are still being implemented in integrated circuits and for ESD protection.

\subsection{Constant-R Bridged T-coils}\label{sect:II}

\begin{figure}
  \centering
  \includegraphics[width=\linewidth]{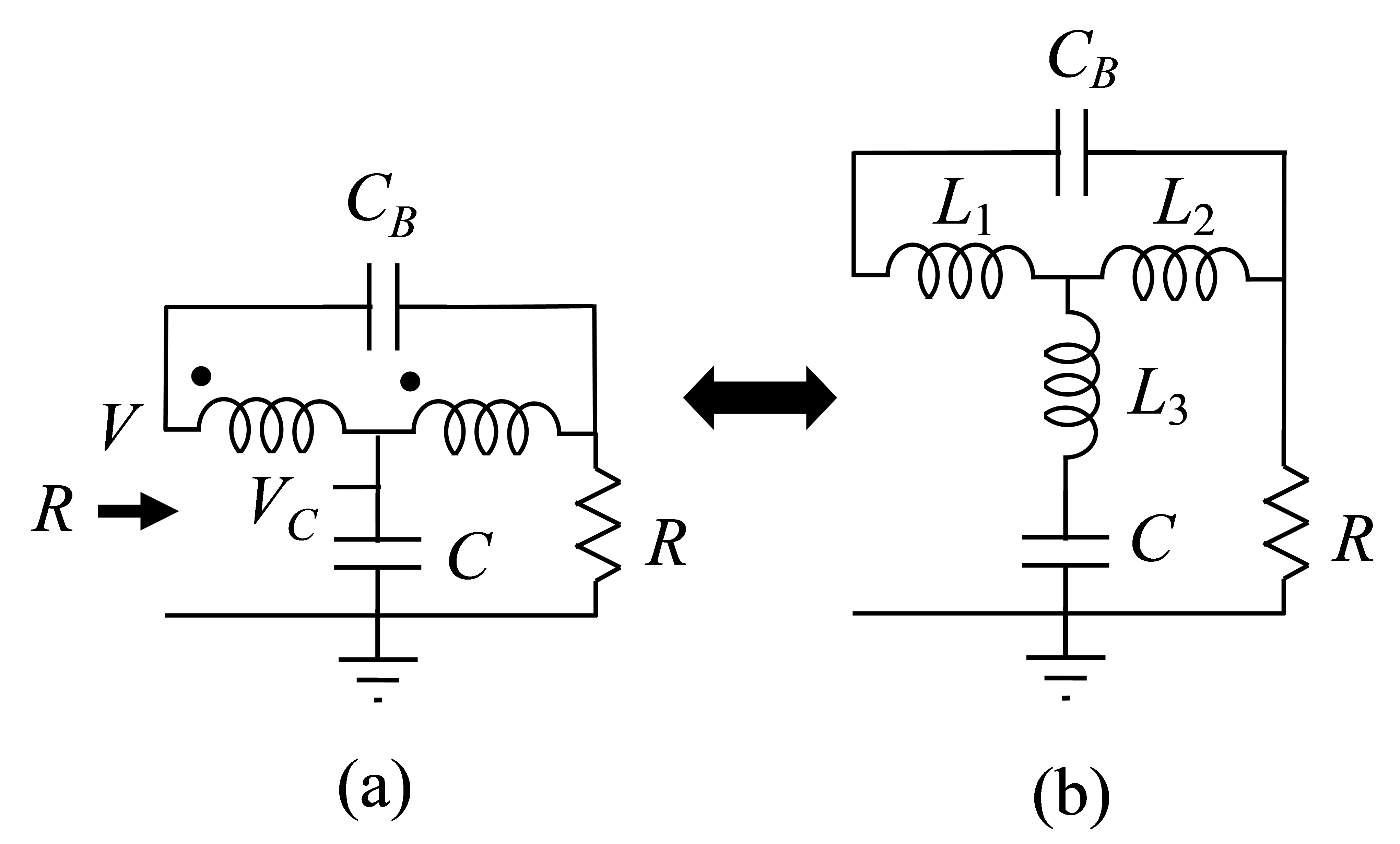}\\
  \caption{Standard constant-R bridged T-coil driving $C$, terminated by $R$, and with a bridging capacitor $C_B$, (a) Coupled transformer model. (b) Equivalent model with three inductors.}\label{fig:Fig1}
\end{figure}

The standard T-coil is shown in Fig. \ref{fig:Fig1}(a). The structure contains a bridging capacitor $C_B$ and is terminated by a resistor $R$. Of interest is the transfer function to a load capacitor $C$. The three-terminal coupled transformer in Fig. 1(a) can be modeled as shown in Fig. \ref{fig:Fig1}(b) by inductors $L_1$, $L_2$, and $L_3$ with well-documented conversions between the two forms. This allows for a wider range of solutions including cases (for some extended configuration values) that make $L_1$ negative. In the standard T-coil, $L_1= L_2$.

\begin{figure}
  \centering
  \includegraphics[width=\linewidth]{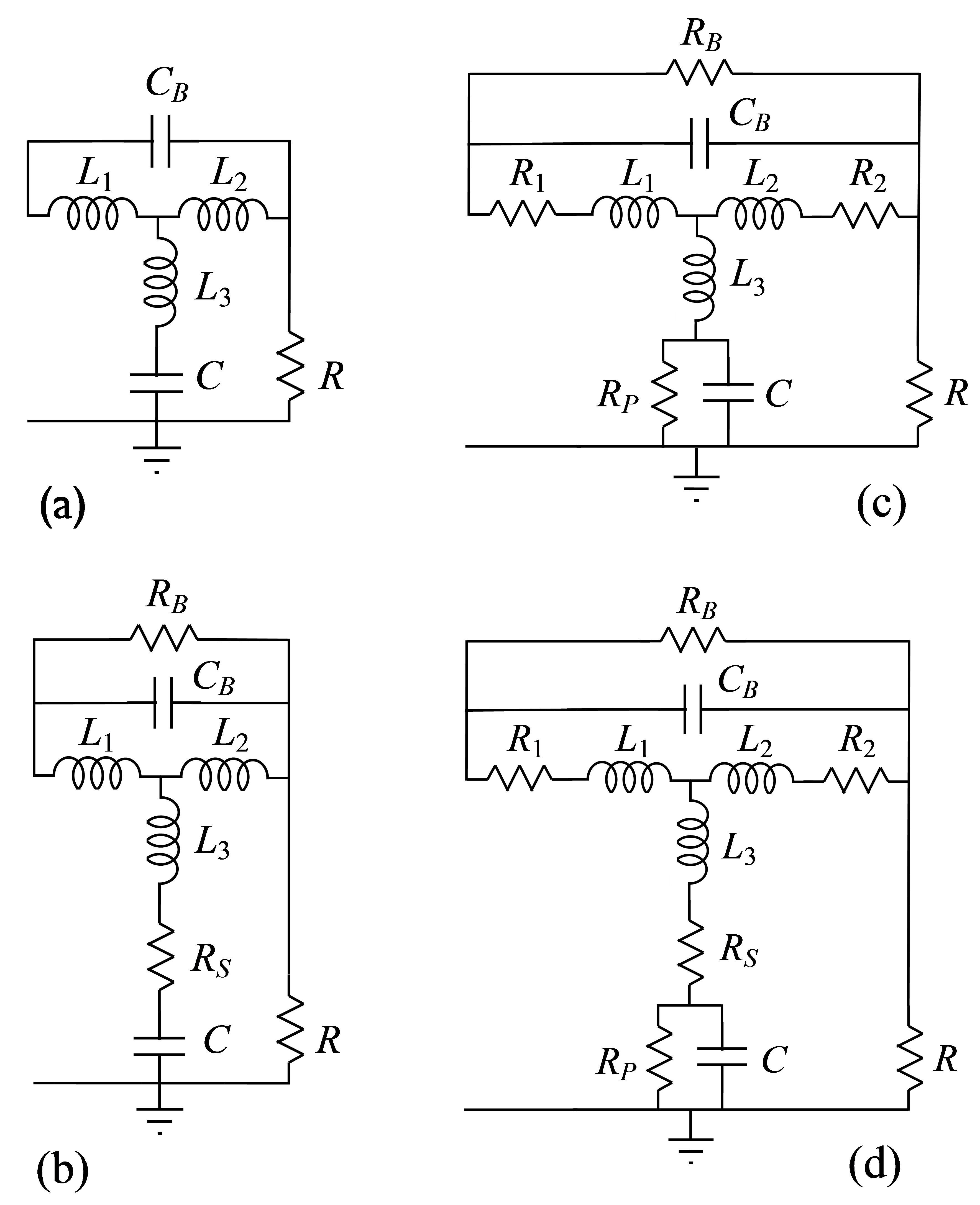}\\
  \caption{Symmetrical T-coils where $L_1 = L_2$ and $R_1 = R_2$ and with calculated $R_B$. (a) Standard configuration (and no $R_B$, $R_1$ and $R_2$). (b) With $R_S$. (c) With $R_P$. (d) With $R_S$ and $R_P$.}\label{fig:Fig2}
\end{figure}

Fig.~\ref{fig:Fig2} shows symmetrical constant-R configurations with Fig.~\ref{fig:Fig2}(a), the standard configuration. Figs.~\ref{fig:Fig2}(b)-\ref{fig:Fig2}(d) adds additional resistor combinations. These extensions add a bridging resistor $R_B$ to allow symmetry with $L_1 = L_2$ and $R_1 = R_2$. Symmetrical structures might be easier to fabricate.

\begin{figure}
  \centering
  \includegraphics[width=\linewidth]{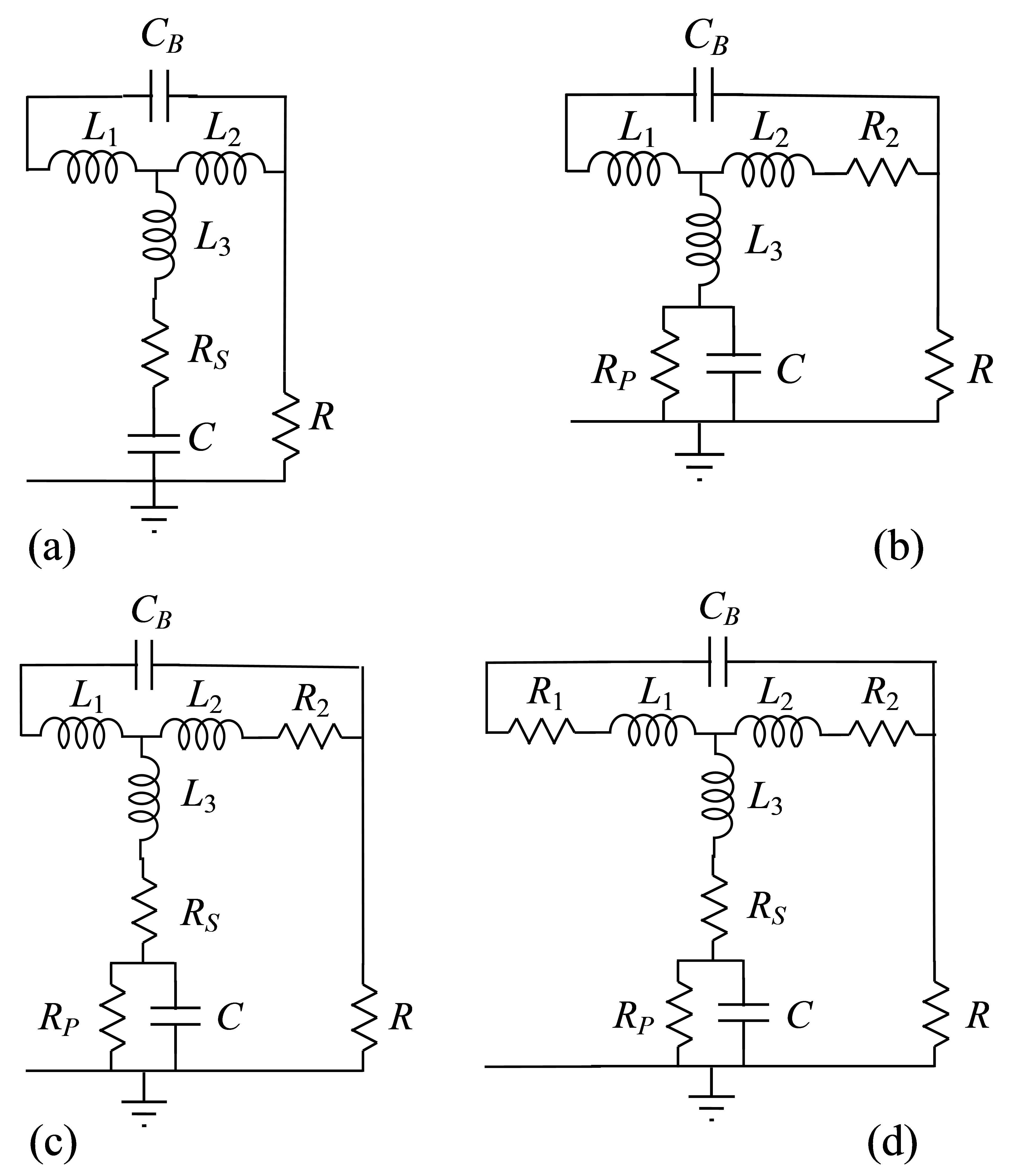}\\
  \caption{Asymmetrical T-coils with added $R_2$ where needed. (a) With $R_S$. (b) With $R_P$. (c) With $R_S$ and $R_P$. (d) Generalized with $R_S$, $R_P$, and selectable $R_1$ and calculated $R_2$ to balance the losses in the $L_1$ and $L_2$ branches.}\label{fig:Fig3}
\end{figure}

Similar to Fig.~\ref{fig:Fig2}, Fig.~\ref{fig:Fig3} shows asymmetrical constant-R T-coils configurations. Fig.~\ref{fig:Fig3}(d) provides a general configuration with a selectable $R_1$ over a limited range as long as $R_2\geq 0$. For example, resistors can be entered proportionally in both inductor branches for physical losses.

The first author derived the formulas in the late 1960s for the series configuration in Fig.~\ref{fig:Fig3}(a) to include bipolar transistor base resistance. Equivalent formulas have been published in later works \cite{Feucht,Addis,Battjes,Hollister,Selmi}. The constant-R component values for Fig.~\ref{fig:Fig3}(b) were revealed in \cite{True}, without the associated transfer function. Later, the first author published equations for generalized configurations including both Fig.~\ref{fig:Fig2}(d) and Fig.~\ref{fig:Fig3}(d) \cite{Ross94}. The asymmetrical configuration did not need a bridging resistor.

\subsection{Constant-Z Bridged-T Networks}\label{sect:III}

\subsubsection{Driving Point Impedance}

A general bridged-T network driving point impedance in Fig.~\ref{fig:Bridged-T} can be derived in symbolic form from loop or nodal equations. Wang algebra simplifies the process with the aforementioned rules. The resulting driving point admittance is set equal to the terminating admittance ${1}/{e}$:
\begin{equation}\label{eq:1}
\begin{split}
    &\frac{1}{Z}=\frac{1}{e}=\frac{N}{D}=\frac{(a+b+c)(b+d+e)}{(a+d)N} \\
    &=\frac{ab+ad+ae+bc+bd+be+cd+ce}{abc+abe+acd+ace+ade+bcd+bde+cde}.
\end{split}
\end{equation}
In \eqref{eq:1} the driving point admittance numerator $N$ is formed by adding the impedances of a set of independent loops except the input loop, and then by multiplying these sums with Wang algebra rules. The denominator $D$ is formed by multiplying $N$ and the input loop sums and applying Wang algebra rules.

With regular algebra, \eqref{eq:1} is simplified to produce symbolic, constant-Z constraints for asymmetrical \eqref{eq:2} and symmetrical \eqref{eq:3} networks as arranged as
\begin{align}
&\left(a+b\right)\left(d-e^2/c\right)+ab+\left(a-b\right)e-e^2 = 0, \label{eq:2}\\
&\ \ \ a=b,\ \ \ \ 2a\left(d-e^2/c\right)+a^2-e^2=0. \label{eq:3}
\end{align}

The component relationships are formed by equating the factors of each power of the Laplace variable $s$ to zero after inserting the actual impedances into \eqref{eq:2} or \eqref{eq:3}. This process yields a set of independent constraints. In some cases the relationships show that the circuit is not suitable due to a mathematical contradiction or to non-realizable (negative) components.

\begin{figure}
  \centering
  \includegraphics[width=\linewidth]{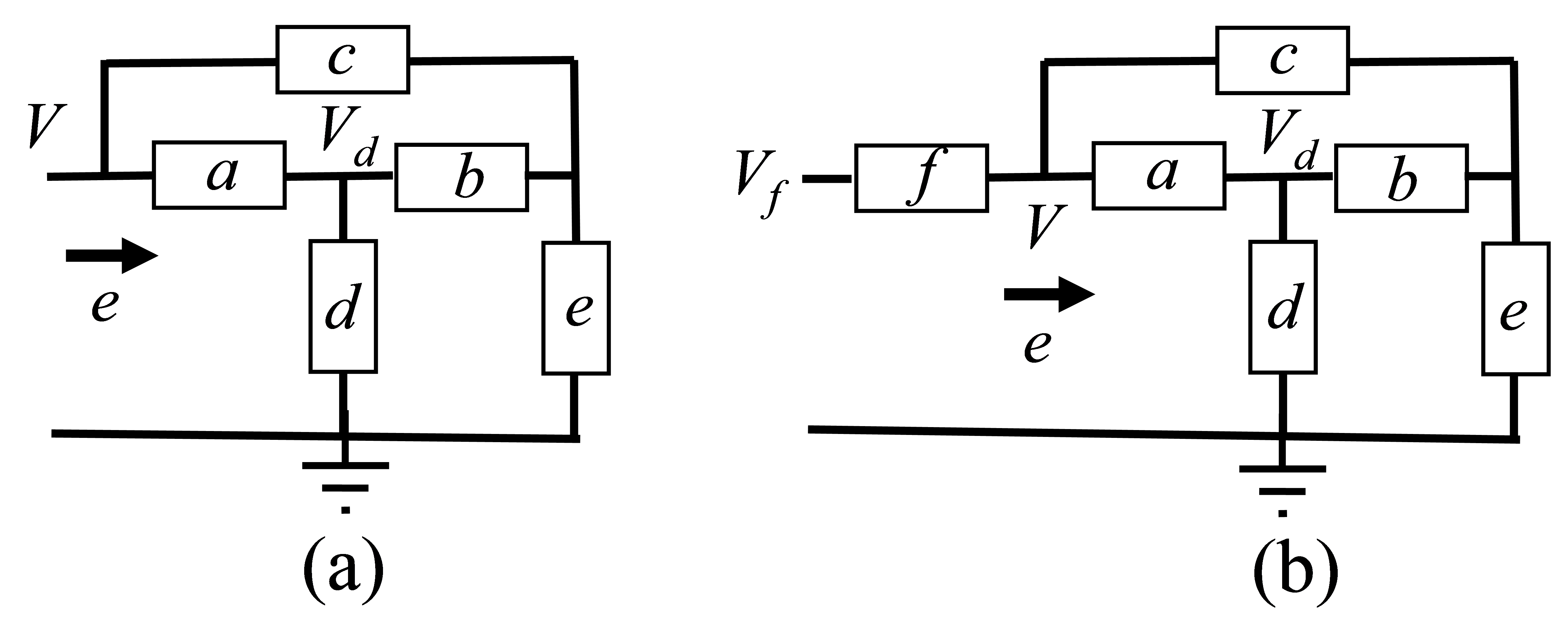}\\
  \caption{Transfer function to $V_d$.  (a) Constant-$Z$ bridged-T network and $V_d$. (b) Equivalent $V_d$ with added  $f$ and $V_f$.}\label{fig:Fig5}
\end{figure}

\begin{figure}
  \centering
  \includegraphics[width=\linewidth]{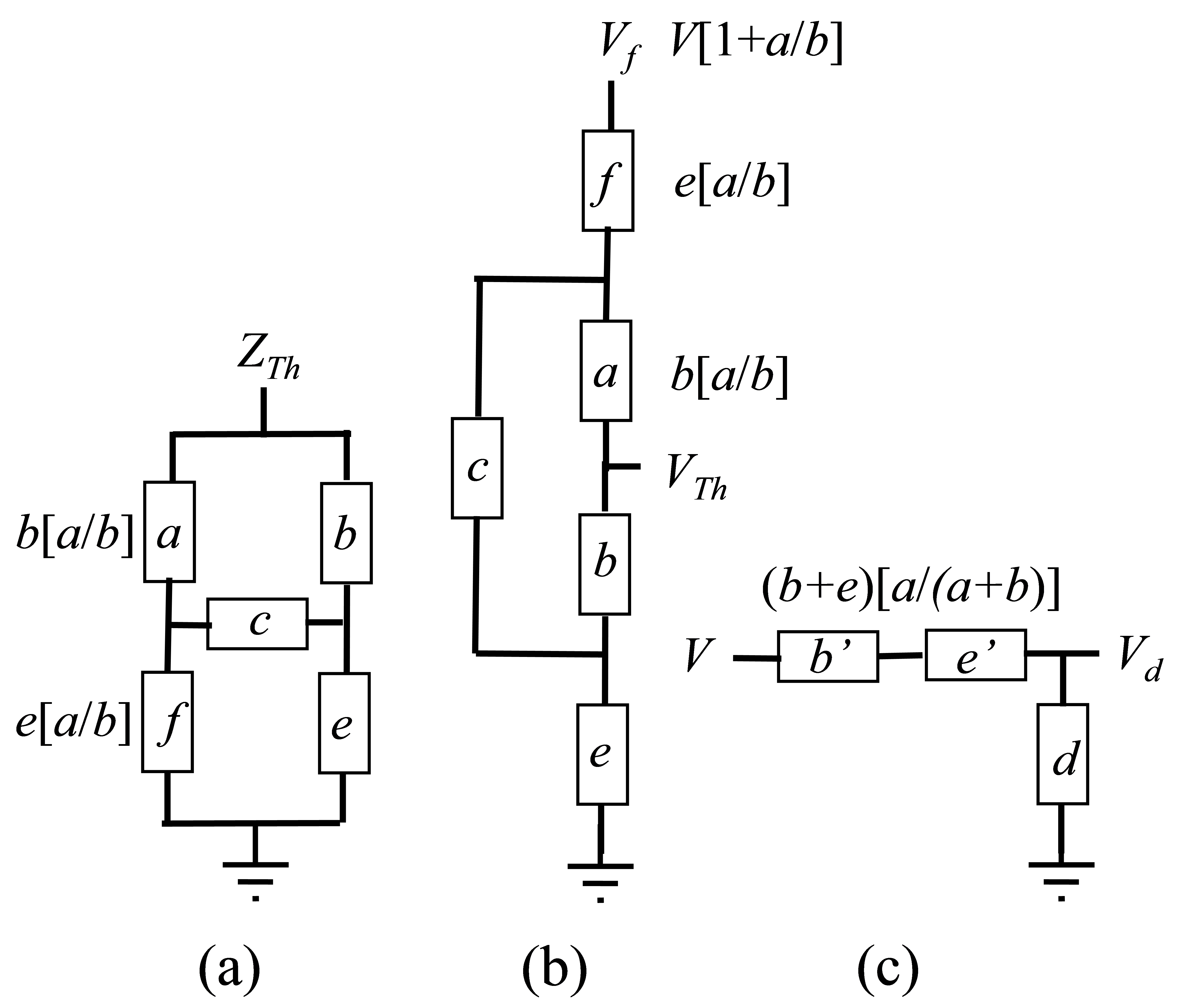}\\
  \caption{Th\'evenin equivalent circuit development to $V_d$. (a) Impedance calculation.  (b) Voltage calculation with scaled-mirror symmetry about $V_{Th}$. (c) Th\'evenin equivalent circuit.}\label{fig:Fig6}
\end{figure}

\subsubsection{Transfer Function}

A method is presented for calculating the transfer function to node $V_d$  in Fig.~\ref{fig:Fig5}(a), if the constant-Z constraint is satisfied. The configuration in Fig.~\ref{fig:Fig5}(b) produces an equivalent result with impedance $f$   and $V_f=V(e+f)/e$. A Th\'evenin equivalent circuit can be derived at node $V_d$  as shown in Fig. 6. Without any loss in generality, the bridge in Fig.~\ref{fig:Fig6}(a) can be balanced by choosing $f=e/(ab)$ to null out the $c$ term.  The Th\'evenin impedance, $Z_{Th}$, becomes $(b+e)\|(a+f)$. As shown in Fig.~\ref{fig:Fig6}(b) this also produces a scaled-mirror symmetrical divider about $V_{Th}$ (which can be shown by splitting $c$ into two impedances, $ca/(a+b)$ and $cb/(a+b)$, and joining them at the $V_{Th}$  node. The values shown in Fig.~\ref{fig:Fig6}(c) are:
\begin{equation}\label{eq:4}
V_{Th}=V,\ \ Z_{Th}=b^\prime+e^\prime=\left(b+e\right)\frac{a}{a+b}.
\end{equation}

This equivalent circuit applies for any constant-Z bridged-T network. For constant-R T-coils, the transfer function of interest is to the capacitor $C$ (usually part of a transistor model) within branch $d$. For symmetrical networks the Th\'evenin impedance reduces to one-half of the impedances of the $b+e$ branches. For asymmetrical networks, the Th\'evenin impedance also produces reduced order transfer functions where applicable (sometimes after applying some constant-Z identities). Fig.~\ref{fig:Fig6}(a) suggests that poles associated with branch $c$  do not appear in the transfer function. The conductance and capacitance elements within branches $a+b$  in \eqref{eq:4} can be cancelled, as shown later.

\subsection{General T-coil Equations}\label{sect:IV}

\subsubsection{Root Locus}

The Th\'evenin equivalent circuit approach produces second order transfer functions for all T-coils in Fig.~\ref{fig:Fig2} and Fig.~\ref{fig:Fig3} and can be parameterized in terms of the bridging capacitor $C_B$  for the complete range of solutions. Also, $G_P=1/R_P$.

\begin{subequations}\label{eq:5}
\begin{align}
\frac{V_C}{V}&=\frac{{1}/{\left(G_P+sC\right)}}{Z_{Th}+R_S+sL_3+\frac{1}{\left(G_P+sC\right)}} \label{eq:5a}\\
\frac{V_C}{V}&=\frac{1}{B_0+B_1s+B_2s^2}=\frac{1}{B_0+\left(D_0+D_1C_B\right)s+D_2C_Bs^2}. \label{eq:5b}
\end{align}
\end{subequations}

\begin{figure}
  \centering
  \includegraphics[width=\linewidth]{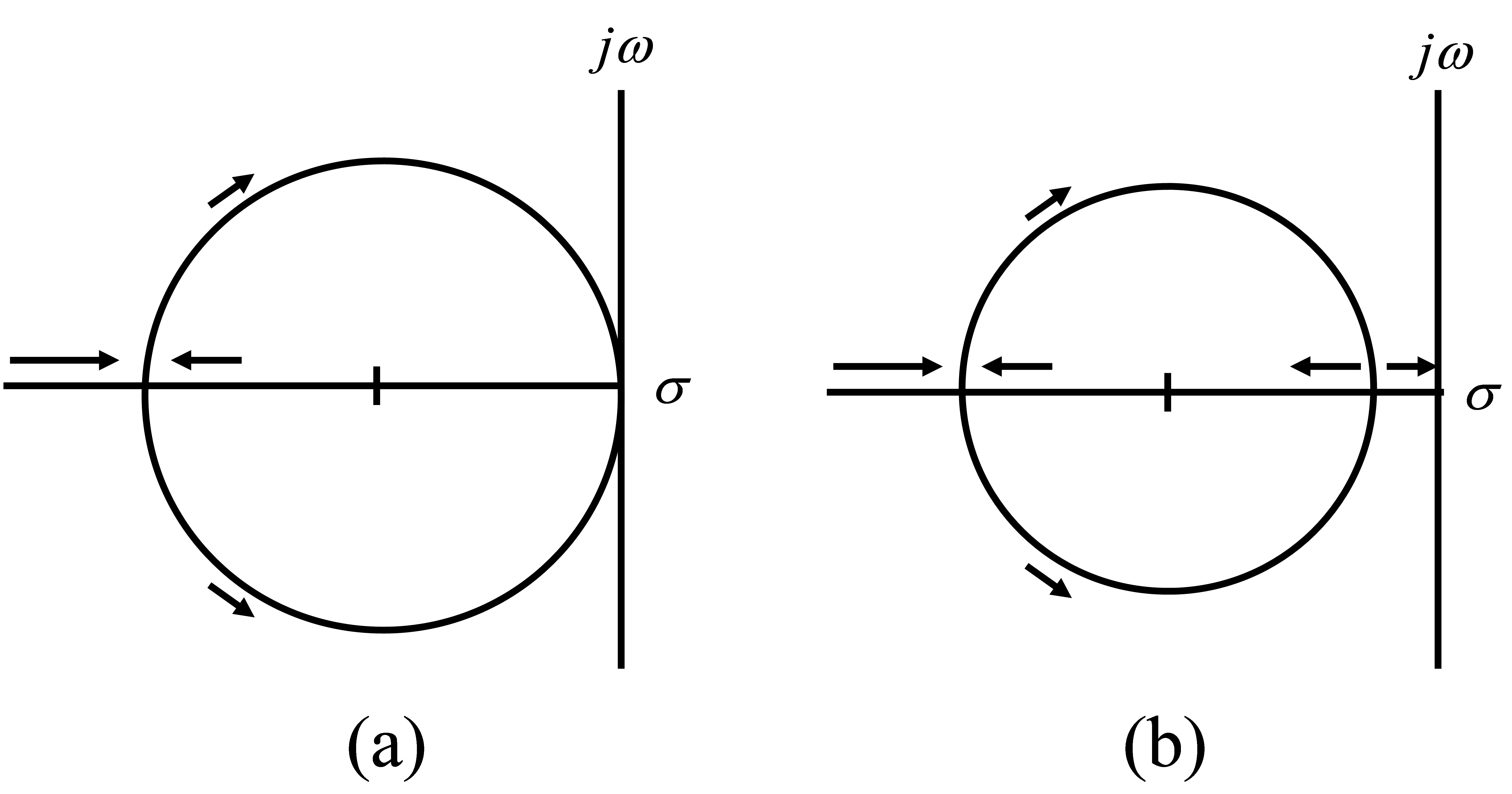}\\
  \caption{Root locus on complex $s=\sigma+j\omega$ plane showing increasing $C_B$. (a) Without $R_P$ ($G_P=0$). (b) With $R_P$.}\label{fig:Fig7}
\end{figure}

In Fig.~\ref{fig:Fig7}, the complex poles of interest are shown for increasing $C_B$ and lie on a circle with
\begin{equation}\label{eq:6}
\mathrm{center}\ =-\frac{B_0}{D_0}, \ \ \    \mathrm{radius}\ =\frac{B_0}{D_0}\sqrt{1-\frac{D_0D_1}{B_0D_2}}.
\end{equation}
Fig.~\ref{fig:Fig6}(b) shows the case where the circle defined by \eqref{eq:6} is to the left of the origin because $D_1>0$ when $R_P$ exists.

A design process is to solve for the poles of \eqref{eq:5} and then find $C_B$ for a specified pole-angle measured from the negative real axis. Both poles are found by solving quadratic equations. Elements and terms for symmetrical and asymmetrical transfer functions in \eqref{eq:5} are given next. The $L_3$  value varies with $C_B$ (and the pole-angle) as shown in \eqref{eq:10} and \eqref{eq:19} below. A final step is to calculate physical inductor values and a coupling coefficient.

\subsubsection{Symmetrical T-coils}

The T-coil equations for the general case in Fig.~\ref{fig:Fig2}(d) are expressed using $G_P=1/R_P$   and $G_B=1/R_B$  to support omitted resistor simplifications with zero-valued $G_P$ and/or $G_B$  elements.

\paragraph{Constant-R Elements}
\begin{align}
R_1&=R_2=R^2G_p/2 \label{eq:7}\\
G_B&=R_S/R^2+G_P/4 \label{eq:8}\\
L_1&=L_2=R^2C/2 \label{eq:9}\\
L_3&=R^2C_B-L_1/2 \label{eq:10}.
\end{align}

\paragraph{Transfer Function Terms}
\begin{align}
	B_0&=1+G_P(R/2+R^2G_P/4+R_S)	\\
  	D_0&=\left(2R+R^2G_P+{4R}_S\right)C/4	\\
	D_1C_B&=R^2G_pC_B	\\
	D_2C_B&=R^2CC_B	.
\end{align}

\subsubsection{Asymmetrical T-coils}

The T-coil equations for the general case in Fig.~\ref{fig:Fig3}(d) are presented.

\paragraph{Constant-R Elements}
\begin{subequations}\label{eq:15}
\begin{align}
	R_2&=\frac{R^2G_P-R_1\left[1+\left(R_S+R\right)G_P\right]}{1+(R_S+{R_1-R)G}_P}	\label{eq:15a}\\
	R_1&=\frac{R^2G_P-R_2\left[1+\left(R_S-R\right)G_P\right]}{1+(R_S+{R_2+R)G}_P}	\label{eq:15b}\\
	R_1+R_2&=\frac{\left(R-R_1\right)^2G_P}{1+\left(R_S+R_1-R\right)G_P}	\label{eq:15c}\\
	R_1+R_2&=\frac{{(R+R_2)}^2G_P}{1+(R_S+R_2+R)G_P}	 \label{eq:15d}
\end{align}
\end{subequations}

\begin{subequations}
\begin{align}
L_TG_P&=\left(L_1+L_2\right)G_P=\left(R_1+R_2\right)C	\label{eq:16a}\\
	L_T&=\frac{\left(R-R_1\right)^2C}{1+\left(R_S+R_1-R\right)G_P}	\label{eq:16b}\\
	L_T&=\frac{{(R+R_2)}^2C}{1+(R_S+R_2+R)G_P}	\label{eq:16c}
\end{align}
\end{subequations}

\begin{align} L_1&=\frac{(R_1-R)(R_1+R_S-R)C}{1+\left(R_1+R_S-R\right)G_P+\sqrt{1+\left(R_1+R_S-R\right)G_P}}	\label{eq:17}\\	L_2&=\frac{(R_2+R)(R_2+R_S+R)C}{1+\left(R_2+R_S+R\right)G_P+\sqrt{1+\left(R_2+R_S+R\right)G_P}}	\label{eq:18} \\
L_3&={R^2C_B-L}_1L_2/(L_1+L_2).	\label{eq:19}
\end{align}

\paragraph{Transfer Function Terms}

\begin{align} 	
B_0&=1+R_SG_P+R_1[1+R+R_S+R_2G_P]R+R_2 \label{eq:20}\\
D_0&=L_2R[1+(R_1+R_S-R)G_P](R-R_1)2 \label{eq:21}\\
D_1C_B&=R^2G_pC_B	\label{eq:22}\\
D_2C_B&=R^2CC_B	\label{eq:23}.
\end{align}

The resistance constraint is implemented in \eqref{eq:15a} after selecting $R_1$; \eqref{eq:15b} is an equivalent arrangement. Summations \eqref{eq:15c} and \eqref{eq:15d} are derived from \eqref{eq:15a} and \eqref{eq:15b} respectively. The summations \eqref{eq:15c} and \eqref{eq:15d} are used in \eqref{eq:16a} to derive \eqref{eq:16b} and \eqref{eq:16c} such that zero-valued elements (besides $R$) are allowed. These forms are useful to derive \eqref{eq:17} and \eqref{eq:18}.  Also, from \eqref{eq:15a}, $R_1$ is limited to a range where $R_2\geq 0$.  If $G_P=0$, $R_1=R_2=0$.

From \eqref{eq:4}, \eqref{eq:5a}, and \eqref{eq:16a}, the transfer function $V_C/V$  in \eqref{eq:5} remains second order, as shown in \eqref{eq:5b}. The $sC+G_P$ term in the product $(sC+G_P ) Z_{Th}$ drops out because in $Z_{Th}$, $a+b=s(L_1+L_2 )+(R_1+R_2)=(R_1+R_2)(sC+G_P ))/G_P$.
If $G_P=0$, then from \eqref{eq:15c} or \eqref{eq:15d}, $(R_1+R_2)/G_P =R^2$.

\subsection{Example}

Table I shows a set of load values and the designs for all T-coils. Elements at and above the Gain line remain unchanged for each Fig.~\ref{fig:Fig2} or Fig.~\ref{fig:Fig3} configuration. Blank entries represent omitted elements for simpler structures. Element values (negative values in parenthesis) for stated pole angles are shown below single solid lines along with corresponding parameters include bandwidth (BW) and bandwidth extension ratio (BWER) relative to a $1+RCs$  single-pole bandwidth.

\begin{table*}
\caption{T-coil example for all configurations. $R=50\ \Omega$, $C=4$ pF, $R_S=10\ \Omega$, $P_P=500\ \Omega$, Pole Angle $=30^{\circ}\ \& \ 45^{\circ}$, $R_1$ is selected for Fig.~\ref{fig:Fig3}(d).} \label{table}
\centering
  \includegraphics[width=\linewidth]{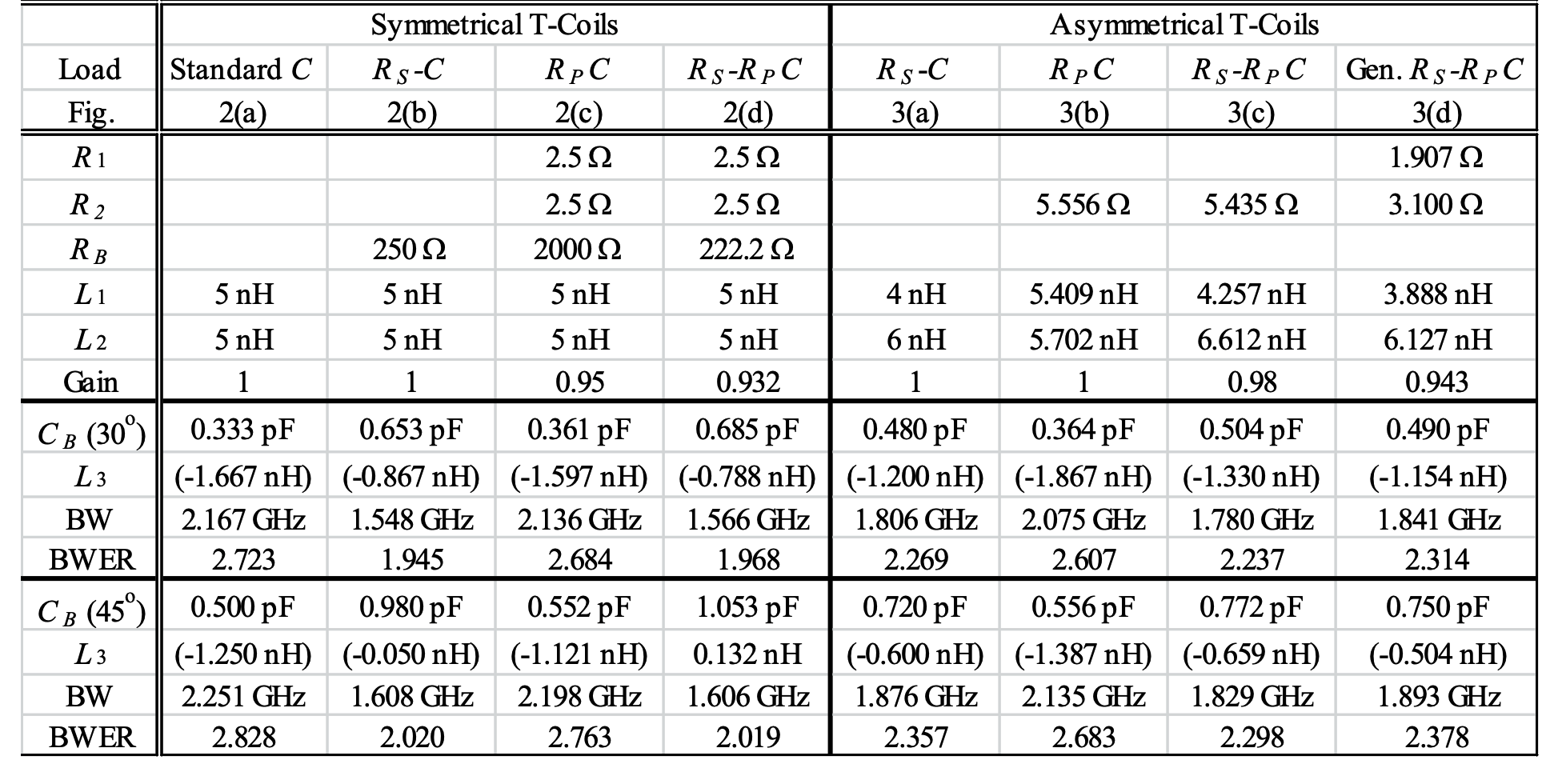}
\end{table*}

\section{Conclusions}\label{sect:conclusion}

In this paper, we have presented a survey of Wang algebra and its applications to T-coils. A Th\'evenin equivalent circuit method simplifies transfer function extraction of constant-Z bridged-T networks. Because of a balanced-bridge process, the Th\'evenin equivalent circuit produces a final result with common poles and zeros already cancelled. Two sets of general constant-R bridged T-coil equations are derived, and a design process is suggested.

Even greater bandwidth extensions have been reported by trading off the constant-R property and controlling other parameters \cite{Shekhar,Walling}. These extensions are outside of the scope of this paper.


\section*{Acknowledgment}

The first author acknowledges C. R. Battjes for his original work in designing T-coils on printed circuits, on hybrid IC substrates, and in ICs; for teaching the T-coil design methods within Tektronix, Inc.; and for inspiring others to advance the technology. The first author also thanks Dr. Ivan Frisch, Professor Emeritus at NYU, for teaching Wang algebra in a course at the University of California, Berkeley. The second author would like to thank Dr. Jin-Hai Guo for providing Ki-Tung Wang's photo and several references.

\appendices

\section{Symmetrical Constant-Z Derivations}

The symmetrical network equations \eqref{eq:7}-\eqref{eq:10} can be derived for the Fig.~\ref{fig:Fig2}(d) elements with the substitutions in \eqref{eq:A1} entered in \eqref{eq:3}:
\begin{equation}\label{eq:A1}
\begin{split}
a&=b=sL_1+R_1, \\
\frac{1}{c}&=\ sC_B+G_B, \\
d&=sL_3+R_S+\frac{1}{\left(sC+G_P\right)},\\
e&=R.
\end{split}
\end{equation}
This substitution expands to the product terms in \eqref{eq:A2}:
\begin{equation}\label{eq:A2}
\begin{split}
&s^3\left(-2CC_BL_1R^2+CL_1^2+2{CL_1L}_3\right)+ \\
&s^2(-2CG_BL_1R^2-2C_B{G_PL}_1R^2-2CC_BR^2R_1 \\
&+G_PL_1^2+2G_PL_1L_3+2CL_1R_1+2CL_3R_1+2CL_1R_S\\
&s^1(-2G_BG_PL_1R^2-2CG_BR^2R_1-2C_BG_PR^2R_1\\
&-CR^2+2G_PL_1R_1+2G_PL_3R_1+CR_1^2 \\
&+2G_PL_1R_S+2CR_1R_S+2L_1)+\\
&s^0(-2G_BG_PR^2R_1-G_PR^2+G_PR_1^2+2G_PR_1R_S+2R_1)\\
&=0.
\end{split}
\end{equation}

One set of steps is outlined.  Commercial or free symbolic mathematical tools can automate the reductions by doing substitutions and cancellations. SageMath (a free, open-source tool created by academia and available in the cloud or for downloading at www.sagemath.org) is used in these steps.

The $s^3$ multiplier terms are set to zero to produce \eqref{eq:10}.  After substituting for $L_3$, the $s^2$ multiplier terms are simplified to yield a preliminary expression for $R_1$ that includes $G_B$ and also to yield \eqref{eq:8} and \eqref{eq:9} directly from the simplified $s^1$  and $s^0$ multiplier terms that are set to zero.  Finally, \eqref{eq:8} is used to simplify the previous $R_1$ result to produce \eqref{eq:7}.

To form the transfer function as a function of $C_B$, \eqref{eq:9} is also used so that $C_B$ can range from zero to infinity.  This provides a complete solution space covering real and complex poles. Values of $C_B$ (and $L_3$) can be determined for complex poles and selected pole angles, as shown in Table I.

\section{Asymmetrical Constant-Z Derivations}

The asymmetrical network equations \eqref{eq:15}-\eqref{eq:19} can be derived for the Fig.~\ref{fig:Fig3}(d) elements with the substitutions in \eqref{eq:B1} entered in \eqref{eq:2}:
\begin{equation}\label{eq:B1}
\begin{split}
a&=sL_1+R_1, \\
b&=sL_2+R_2, \\
\frac{1}{c}&=\ sC_B, \\
d&=sL_3+R_S+\frac{1}{\left(sC+G_P\right)},\\
e&=R.
\end{split}
\end{equation}

This substitution expands to the product terms in \eqref{eq:B2}:
\begin{equation}\label{eq:B2}
\begin{split}
&s^3(-CC_BL_1R^2-CC_BL_2R^2+CL_1L_2+{CL_1L}_3)\\
&+{CL_2L}_3)+\\
&s^2(-{C_BG}_PL_1R^2-{C_BG}_PL_2R^2-CC_BR^2R_1\\
&-CC_BR^2R_2+G_PL_1L_2+G_PL_1L_3\ {+\ G}_PL_2L_3\\
&+CL_1R+CL_2R+CL_2R_1+CL_3R_1+\ CL_1R_2\\
&+CL_3R_2+CL_1R_S+CL_2R_S)+\\
&s^1(-C_BG_P{R^2R}_1-C_BG_P{R^2R}_2+G_PL_1R \\
&+G_PL_2R+CR^2+G_PL_2R_1+G_PL_3R_1+CRR_1\\
&+G_PL_1R_2+G_PL_3R_2-CRR_2\\
&+CR_1R_2+G_PL_1R_S+G_PL_2R_S\\
&+CR_1R_S+CR_2R_S+L_1+L_2)+\\
&s^0({-\ G}_PR^2+G_PRR_1-G_PRR_2+G_PR_1R_2\\
&+G_PR_1R_S+G_PR_2R_S+R_1+R_2)=0.
\end{split}
\end{equation}
One set of steps is briefly outlined. Some of the later steps are done manually. The $s^3$ multiplier terms are set to zero to produce \eqref{eq:19}. The $s^0$  multiplier terms are set to zero to produce \eqref{eq:15a} and \eqref{eq:15b}. After replacing $L_3$ with \eqref{eq:19}, all terms with $L_3$ and $C_B$ elements are eliminated in a modified expansion.

In this expansion, the $s^2$ multiplier terms yield \eqref{eq:B3} after they are set to zero and divided by $C$:
\begin{equation}\label{eq:B3}
L_2^2\left(R_S+R_1-R\right)+2L_1L_2R_S+L_1^2\left(R_S+R_2+R\right)=0.
\end{equation}
The $s^1$ multipliers are simplified with two substitutions. The terms that are multiplied by $G_P$ are identical to the terms in the left side of \eqref{eq:B3} and are removed.  The remaining terms that are multiplied by $C(L_1+L_2)$ are replaced because they are identical to the first six terms of the $s^0$ multiplier in \eqref{eq:B2} after these six terms are divided by $-G_P$. The mathematical process produces $(L_1+L_2)(R_1+R_2)/G_P$, which is set equal to a remaining $(L_1+L_2)^2$. Equation \eqref{eq:16a} follows after dividing both sides by $(L_1+L_2)$.

A quadratic equation formula is used to solve \eqref{eq:B3} for either $L_2$ in \eqref{eq:B4} or $L_1$ in \eqref{eq:B5} in terms of the other inductor. Each solution has the same six product terms under the radical as the first six $s^0$ multiplier terms divided by $-G_P$  in \eqref{eq:B2}. These terms can be simplified as described above to form $\sqrt{(R_1+R_2)/G_P}$.
\begin{align}
	L_2&=L_T-L_1=L_1\frac{-R_S+\sqrt{(R_1+R_2){/G}_P}}{R_S+R_1-R},	\label{eq:B4}\\
	L_1&=L_T-L_2=L_2\frac{-R_S+\sqrt{(R_1+R_2){/G}_P}}{R_S+R_2+R}.	\label{eq:B5}
\end{align}
After doing some mathematical simplification, \eqref{eq:17} follows from \eqref{eq:B4}, \eqref{eq:15c}, and \eqref{eq:16b}. Equation \eqref{eq:18} follows from \eqref{eq:B5}, \eqref{eq:15d}, and \eqref{eq:16c}.

Equations \eqref{eq:17} and \eqref{eq:18} are formed with positive radicals in \eqref{eq:B4} and \eqref{eq:B5}. The negative radical solutions produce divide by zero cases if $G_P=0$.

As done in the symmetrical case, \eqref{eq:19} is substituted to express the transfer denominator \eqref{eq:20}-\eqref{eq:23} as a function of $C_B$.

\footnotesize
\bibliographystyle{IEEEtran}
\bibliography{IEEEabrv,WangAlgebra}

\begin{thebibliography}{10}
\providecommand{\url}[1]{#1}
\csname url@samestyle\endcsname
\providecommand{\newblock}{\relax}
\providecommand{\bibinfo}[2]{#2}
\providecommand{\BIBentrySTDinterwordspacing}{\spaceskip=0pt\relax}
\providecommand{\BIBentryALTinterwordstretchfactor}{4}
\providecommand{\BIBentryALTinterwordspacing}{\spaceskip=\fontdimen2\font plus
\BIBentryALTinterwordstretchfactor\fontdimen3\font minus
  \fontdimen4\font\relax}
\providecommand{\BIBforeignlanguage}[2]{{%
\expandafter\ifx\csname l@#1\endcsname\relax
\typeout{** WARNING: IEEEtran.bst: No hyphenation pattern has been}%
\typeout{** loaded for the language `#1'. Using the pattern for}%
\typeout{** the default language instead.}%
\else
\language=\csname l@#1\endcsname
\fi
#2}}
\providecommand{\BIBdecl}{\relax}
\BIBdecl

\bibitem{Wang}
K.~T. Wang, ``On a new method of analysis of electrical networks,''
  \emph{Memoir of National Research Institute of Engineering, Academia Sinica},
  vol.~2, pp. 1--11, 1934.

\bibitem{Ting}
S.-L. Ting, ``On the general properties of electrical network determinants,''
  \emph{Chinese Journal of Physics}, vol.~1, pp. 18--40, 1935.

\bibitem{Tsai}
C.-T. Tsai, ``Short cut methods for expanding the determinants involved in
  network problems,'' \emph{Chinese Journal of Physics}, vol.~3, pp. 148--181,
  1939.

\bibitem{Chow}
W.-L. Chow, ``On electric networks,'' \emph{J. Chinese Math. Soc.}, vol.~2, pp.
  321--339, 1940.

\bibitem{KU}
Y.-H. Ku, ``R\'esum\'e of {Maxwell's and Kirchhoff's} rules for network
  analysis,'' \emph{Journal of the Franklin Institute}, vol. 253, no.~3, pp.
  211--224, 1952.

\bibitem{BELLERT}
S.~Bellert, ``Topological analysis and synthesis of linear systems,''
  \emph{Journal of the Franklin Institute}, vol. 274, no.~6, pp. 425--443,
  1962.

\bibitem{Ross94}
B.~Ross, ``Generalization of {T-coil} equations,'' in \emph{3rd
  Electrotechnical and Comput. Sci. Conf.}, Portoroz, Slovenia, 1994, pp.
  39--43.

\bibitem{Ross07}
\BIBentryALTinterwordspacing
------, ``Wang algebra and interconnects,'' in \emph{Asian IBIS Summit},
  Beijing, China, 2007. [Online]. Available:
  \url{https://ibis.org/summits/sep07a/ross.pdf}
\BIBentrySTDinterwordspacing

\bibitem{Chen-Graph}
W.-K. Chen, \emph{Graph Theory and Its Engineering Applications}.\hskip 1em
  plus 0.5em minus 0.4em\relax Singapore: World Scientific, 1997.

\bibitem{Duffin}
R.~J. Duffin, ``An analysis of the {Wang} algebra of networks,'' \emph{Trans.
  Amer. Math. Soc.}, vol.~93, pp. 114--131, 1959.

\bibitem{Duffin-Morley}
R.~Duffin and T.~Morley, ``Wang algebra and matroids,'' \emph{IEEE Transactions
  on Circuits and Systems}, vol.~25, no.~9, pp. 755--762, 1978.

\bibitem{Duffin74}
R.~J. Duffin, ``Some problems of mathematics and science,'' \emph{Bulletin of
  the American Mathematical Society}, vol.~80, pp. 1053--1070, 1974.

\bibitem{Guo2015}
\BIBentryALTinterwordspacing
J.-H. Guo, ``{K. T. Wang}: First {Chinese} scholar to publish a mathematical
  paper in international journals,'' 2015, (in Chinese), Institute for the
  History of Natural Sciences, Chinese Academy of Sciences. [Online].
  Available:
  \url{http://www.ihns.cas.cn/kxcb_new/kpwz_new/201602/t20160229_4538251.html}
\BIBentrySTDinterwordspacing

\bibitem{Guo15}
------, ``From the {Nine Chapters} to ancient {Chinese} mathematics: An
  interpretation of {Ki-Tung Wang}'s letter to {Yan Li},'' \emph{Journal of
  Gaungxi University of Nationalities (Natural Science Edition)}, vol.~21,
  no.~1, pp. 14--18, 2015, (in Chinese).

\bibitem{KTWang-Quaternion}
K.~T. Wang, ``The differentiation of quaternion function,'' \emph{Proceedings
  of the Royal Irish Academy}, vol.~29, pp. 73--80, 1911/1912.

\bibitem{Guo02}
J.-H. Guo, ``{K. T. Wang} and the differentiation of quaternion functions,''
  \emph{China Historical Materials of Science and Technology}, vol.~23, no.~1,
  pp. 65--70, 2002, (in Chinese).

\bibitem{Guo03b}
------, ``{Ki-Tung Wang}'s new method for the analysis of electric network and
  its scientific influence,'' \emph{China Historical Materials of Science and
  Technology}, vol.~24, no.~4, pp. 312--319, 2002, (in Chinese).

\bibitem{Guo03}
------, ``An inquiry of mathematics teaching in the {Imperial Tungwen
  College},'' \emph{Journal of Gaungxi University of Nationalities (Natural
  Science Edition)}, vol.~22, no. Supplementary Issue, pp. 47--60, 2003, (in
  Chinese).

\bibitem{WangBook}
\BIBentryALTinterwordspacing
K.~T. Wang, \emph{Comparative Study of Buddhism and Sciences}.\hskip 1em plus
  0.5em minus 0.4em\relax Shanghai Buddhism Press, 1933; reprinted by Shanxi
  People's Press, 2014, (in Chinese). [Online]. Available:
  \url{http://www.nnycjd.com/jsrw/wjt/8000.html}
\BIBentrySTDinterwordspacing

\bibitem{Lee}
T.~H. Lee, \emph{Planar Microwave Engineering: A Practical Guide to Theory,
  Measurements, and Circuits}.\hskip 1em plus 0.5em minus 0.4em\relax New York,
  NY: Cambridge University Press, 2004.

\bibitem{Ginzton}
E.~Ginzton, W.~Hewlett, J.~Jasberg, and J.~Noe, ``Distributed amplification,''
  \emph{Proceedings of the IRE}, vol.~36, no.~8, pp. 956--969, 1948.

\bibitem{HP}
N.~B. Schrock, ``A new amplifier for milli-microsecond pulses,'' \emph{HP
  Journal}, vol.~1, no.~1, pp. 9--13, Sept. 1949.

\bibitem{Staric}
P.~Stari\v{c} and E.~Margan, \emph{Wideband Amplifiers}, 4th~ed.\hskip 1em plus
  0.5em minus 0.4em\relax Dordrecht, The Netherlands: Springer, 2015, pp.
  2.35-2.50.

\bibitem{Paramesh}
J.~Paramesh and D.~J. Allstot, ``Analysis of the bridged {T-Coil} circuit using
  the extra-element theorem,'' \emph{IEEE Transactions on Circuits and Systems
  II: Express Briefs}, vol.~53, no.~12, pp. 1408--1412, 2006.

\bibitem{Roy}
S.~C.~D. Roy, ``Comments on ``{Analysis} of the bridged {T-coil} circuit using
  the extra-element theorem",'' \emph{IEEE Transactions on Circuits and Systems
  II: Express Briefs}, vol.~54, no.~8, pp. 673--674, 2007.

\bibitem{Feucht}
D.~L. Feucht, \emph{Handbook of Analog Circuit Design}.\hskip 1em plus 0.5em
  minus 0.4em\relax San Diego, CA: Academic Press, 1990.

\bibitem{Addis}
J.~L. Addis, ``Good engineering and fast vertical amplifiers,'' in \emph{Analog
  Circuit Design: Art, Science, and Personalities}, J.~Williams, Ed.\hskip 1em
  plus 0.5em minus 0.4em\relax Boston, MA, Butterworth-Heinemann, 1991, pp.
  107--122.

\bibitem{Battjes}
C.~R. Battjes, ``Who wakes the bugler?'' in \emph{The Art and Science of Analog
  Circuit Design}, J.~Williams, Ed.\hskip 1em plus 0.5em minus 0.4em\relax
  Boston, MA, Butterworth-Heinemann, 1995, pp. 121--138.

\bibitem{Hollister}
A.~L. Hollister, \emph{Wideband Amplifier Design}.\hskip 1em plus 0.5em minus
  0.4em\relax Raleigh, NC, SciTech Publishing, 2007, pp. 20-46.

\bibitem{Selmi}
L.~Selmi, D.~Estreich, and B.~Ricco, ``Small-signal {MMIC} amplifiers with
  bridged {T-coil} matching networks,'' \emph{IEEE Journal of Solid-State
  Circuits}, vol.~27, no.~7, pp. 1093--1096, 1992.

\bibitem{True}
T.~T. True, ``Bridged {T} termination network,'' U.S. Patent 3 155 927, Nov.
  1964.

\bibitem{Galal}
S.~Galal and B.~Razavi, ``Broadband {ESD} protection circuits in {CMOS}
  technology,'' \emph{IEEE Journal of Solid-State Circuits}, vol.~38, no.~12,
  pp. 2334--2340, 2003.

\bibitem{Pillai}
E.~Pillai and J.~Weiss, ``Novel {T-coil} structure and implementation in a
  {6.4-Gbs CMOS} receiver to meet return loss specifications,'' in
  \emph{Electron. Components and Tech. Conf.}, 2007, pp. 147--153.

\bibitem{Pillai-patent}
E.~R. Pillai, L.~L. Hsu, W.~Sauter, and D.~W. Storaska, ``On-pad broadbanding
  matching network,'' U.S. Patent 7 265 433 B2, Sept. 2007.

\bibitem{Razavi}
B.~Razavi, ``The bridged {T}-coil [a circuit for all seasons],'' \emph{IEEE
  Solid-State Circuits Magazine}, vol.~7, no.~4, pp. 9--13, 2015.

\bibitem{RAKO}
\BIBentryALTinterwordspacing
P.~Rako, ``What's all this {T}-coil stuff, anyhow?'' Electronic Design, April
  2019. [Online]. Available:
  \url{https://www.electronicdesign.com/technologies/analog/article/21807815/whats-all-this-tcoil-stuff-anyhow}
\BIBentrySTDinterwordspacing

\bibitem{HuGZ}
G.-Z. Hu, \emph{Applied Modern Algebra}.\hskip 1em plus 0.5em minus 0.4em\relax
  Tsinghua University Press, 1993, (in Chinese).

\bibitem{Nine-Chapters}
K.~Shen, J.~N. Crossley, A.~W.-C. Lun, and H.~Liu, \emph{The Nine Chapters on
  the Mathematical Art: Companion and Commentary}.\hskip 1em plus 0.5em minus
  0.4em\relax Oxford University Press, 1999.

\bibitem{Shekhar}
S.~Shekhar, J.~Walling, and D.~Allstot, ``Bandwidth extension techniques for
  {CMOS} amplifiers,'' \emph{IEEE Journal of Solid-State Circuits}, vol.~41,
  no.~11, pp. 2424--2439, 2006.

\bibitem{Walling}
J.~S. Walling, S.~Shekhar, and D.~J. Allstot, ``Wideband {CMOS} amplifier
  design: Time-domain considerations,'' \emph{IEEE Transactions on Circuits and
  Systems I: Regular Papers}, vol.~55, no.~7, pp. 1781--1793, 2008.

\end{thebibliography}

%





\end{document}